\documentclass[11pt,a4paper]{article}
\usepackage[affil-it]{authblk}
\usepackage[ansinew]{inputenc}
\usepackage{amsmath, amssymb, mathrsfs, amsthm, amsfonts}
\usepackage{latexsym}
\usepackage[pdftex]{hyperref}
\usepackage{a4wide}
\usepackage{color}
\usepackage{graphicx}

\newtheorem{theorem}{Theorem}[section]
\newtheorem{corollary}[theorem]{Corollary}
\newtheorem{lemma}[theorem]{Lemma}
\newtheorem{proposition}[theorem]{Proposition}

\theoremstyle{definition}
\newtheorem{definition}[theorem]{Definition}
\theoremstyle{remark}
\newtheorem*{example}{Example}
\newtheorem*{remark}{Remark}

\title{The Varchenko Determinant for Apartments}

\author{Hery Randriamaro
\thanks{This research was funded by my mother \\
Lot II B 32 bis Faravohitra, 101 Antananarivo, Madagascar \\
e-mail: \texttt{hery.randriamaro@gmail.com}}}

\begin{document}

\maketitle

\begin{abstract}
\noindent Varchenko introduced a distance function on chambers of hyperplane arrangements that he called quantum bilinear form. That gave rise to a determinant indexed by chambers whose entry in position $(C,D)$ is the distance between $C$ and $D$: that is the Varchenko determinant. He showed that that determinant has a nice factorization. Later, Aguiar and Mahajan defined a generalization of the quantum bilinear form, and computed the Varchenko determinant given rise by that generalization for central hyperplane arrangements and their cones. This article takes inspiration from their proof strategy to compute the Varchenko determinant given rise by their distance function for apartment of hyperplane arrangements. Those latter are in fact realizable conditional oriented matroids. 

\bigskip 


\smallskip

\end{abstract}

\section{Introduction}

\noindent Let $a_1, \dots, a_n, b \in \mathbb{R}$ such that $(a_1, \dots, a_n) \neq (0, \dots, 0)$. Recall that a hyperplane in $\mathbb{R}^n$ is a $(n-1)$-dimensional affine subspace $\big\{(x_1, \dots, x_n) \in \mathbb{R}^n \ \big|\ a_1 x_1 + \dots + a_n x_n = b\big\}$, and a hyperplane arrangement is a finite set of hyperplanes.

\noindent Denote by $\overline{S}$ the closure of a subset $S \subseteq \mathbb{R}^n$. To every hyperplane $H$ can be associated two connected open half-spaces $H^+$ and $H^-$ such that $H^+ \sqcup H^0 \sqcup H^- = \mathbb{R}^n$ and $\overline{H^+} \cap \overline{H^-} = H^0$, letting $H^0 := H$. A face of a hyperplane arrangement $\mathcal{A}$ in $\mathbb{R}^n$ is a nonempty subset $F$ in $\mathbb{R}^n$ having the form
$$F:= \bigcap_{H \in \mathcal{A}} H^{\epsilon_H(F)} \quad \text{with} \quad \epsilon_H(F) \in \{+,0,-\}.$$
Let $F_{\mathcal{A}}$ be the set formed by the faces of $\mathcal{A}$. It is a poset with partial order $\preceq$ defined by $$F \preceq G \quad \Longleftrightarrow \quad \forall H \in \mathcal{A}:\ \epsilon_H(F) \neq 0 \, \Rightarrow \, \epsilon_H(F) = \epsilon_H(G).$$

\noindent The sign sequence of a face $F \in F_{\mathcal{A}}$ is $\epsilon_{\mathcal{A}}(F) := \big(\epsilon_H(F)\big)_{H \in \mathcal{A}}$. A chamber is a face whose sign sequence contains no $0$. Denote the set formed by the chambers of $\mathcal{A}$ by $C_{\mathcal{A}}$.

\begin{definition}
	Let $\mathcal{K}$ be a subset of a hyperplane arrangement $\mathcal{A}$ in $\mathbb{R}^n$. An \textbf{apartment} of $\mathcal{A}$ is a chamber of $\mathcal{K}$. Denote the set formed by the apartments of $\mathcal{A}$ by $K_{\mathcal{A}}$.
\end{definition}

\noindent Consider an apartment $K \in K_{\mathcal{A}}$. The sets formed by the faces and the chambers in $K$ are respectively $$F_{\mathcal{A}}^K := \{F \in F_{\mathcal{A}}\ |\ F \subseteq K\} \quad \text{and} \quad C_{\mathcal{A}}^K := C_{\mathcal{A}} \cap F_{\mathcal{A}}^K.$$

\noindent Bandelt, Chepoi, and Knauer introduced in 2015 a combinatorial object called conditional oriented matroid or COM \cite{BaChKn}. COMs are common generalizations of oriented matroids and lopsided sets. The former are abstractions for directed graphs and central pseudohyperplane arrangements, while the latter are common generalizations of antimatroids and median graphs. Apartment of hyperplane arrangement is in fact the model of realizable COMs described by Bandelt et al. \cite[§~1.2]{BaChKn}. It generalizes realizability of oriented and affine oriented matroids on one hand and that of lopsided sets on the other hand. 

\smallskip

\noindent Assign a variable $h_H^{\varepsilon}$, $\varepsilon \in \{+,-\}$, to every half-space $H^{\varepsilon}$, $H \in \mathcal{A}$. We work with the polynomial ring $R_{\mathcal{A}} := \mathbb{Z}\big[h_H^{\varepsilon}\ \big|\ \varepsilon \in \{+,-\},\, H \in \mathcal{A}\big]$ in variables $h_H^{\varepsilon}$. For two chambers $C,D \in C_{\mathcal{A}}$, the set of half-spaces containing $C$ but not $D$ is $$\mathscr{H}(C,D) := \big\{H^{\epsilon_H(C)}\ \big|\ H \in \mathcal{A},\, \epsilon_H(C) = - \epsilon_H(D)\big\}.$$

\noindent The distance function $\mathrm{v}: C_{\mathcal{A}} \times C_{\mathcal{A}} \rightarrow R_{\mathcal{A}}$ of Aguiar and Mahajan is defined by  $$\mathrm{v}(C,C) = 1 \quad \text{and} \quad \mathrm{v}(C,D) = \prod_{H^{\varepsilon} \in \mathscr{H}(C,D)} h_H^{\varepsilon}\,\ \text{if}\,\ C \neq D.$$

\noindent Although $\mathrm{v}$ is not symmetric, we keep the name distance function as its authors called it so \cite[§~8.1]{AgMa}. The quantum bilinear form is $\mathrm{v}$ with the restriction $h_H^+ = h_H^-$.

\begin{definition}
The \textbf{Varchenko determinant} for an apartment $K$ of a hyperplane arrangement $\mathcal{A}$ in $\mathbb{R}^n$ is $\big|\mathrm{v}(D,C)\big|_{C,D \in C_{\mathcal{A}}^K}$.
\end{definition}

\noindent That determinant was originally defined with the quantum bilinear form and for hyperplane arrangement by Varchenko \cite[§ 1]{Va2}. But, it already appeared earlier in the implicit form of the determinant of a symmetric bilinear form on a Verma module over a $\mathbb{C}$-algebra \cite[§ 1]{ScVa}. Furthermore, it plays a key role to prove the realizability of variant models of quon algebras like a deformed quon algebra \cite[Theorem~4.2]{Ra}, and a multiparametric quon algebra \cite[Proposition~2.1]{Ra2}. The quon algebra is an approach to particle statistics in order to provide a theory in which the Pauli exclusion principle and Bose statistics are violated.

\noindent The Varchenko matrix of a hyperplane arrangement $\mathcal{A}$ is $\displaystyle V_{\mathcal{A}} := \big(\mathrm{v}(D,C)\big)_{C,D \in C_{\mathcal{A}}}$. That matrix has been investigated from several angles. For the quantum bilinear form used on hyperplane arrangements $\mathcal{S}$ in semigeneral position, Gao and Zhang computed the diagonal form of $V_{\mathcal{S}}$ \cite[Theorem~2]{GaZh}. Recently, Olzhabayev and Zhang extended that result to central pseudohyperplane arrangements in semigeneral position \cite[Theorem~1.1]{OlZh}. Denham and Hanlon studied the Smith normal of the Varchenko matrix for the restriction $h_H^+ = h_H^- = q$ \cite[Theorem~3.3]{DeHa}. And for the same restriction, Hanlon and Stanley computed the nullspace of the Varchenko matrix of braid arrangements \cite[Theorem~3.3]{HaSt}.

\bigskip

\noindent The centralization of a hyperplane arrangement $\mathcal{A}$ to a face $F \in F_{\mathcal{A}} \setminus C_{\mathcal{A}}$ is the hyperplane arrangement $\mathcal{A}_F := \{H \in \mathcal{A}\ |\ F \subseteq H\}$. The weight of $F \in F_{\mathcal{A}} \setminus C_{\mathcal{A}}$ is the monomial
$$\displaystyle \mathrm{b}_F := \prod_{H \in \mathcal{A}_F} h_H^+ \, h_H^- \in R_{\mathcal{A}}.$$
Choose a hyperplane $H \in \mathcal{A}_F$. The multiplicity of $F \in F_{\mathcal{A}} \setminus C_{\mathcal{A}}$ is the integer
$$\displaystyle \beta_F := \frac{\#\{C \in C_{\mathcal{A}}\ |\ \overline{C} \cap H = F\}}{2}.$$ 
We will see in Section~\ref{SecVar} that $\beta_F$ is independent of that chosen $H$. We can now state the main result of this article.

\begin{theorem} \label{MTh2}
Let $\mathcal{A}$ be a hyperplane arrangement in $\mathbb{R}^n$, and $K \in K_{\mathcal{A}}$. Then, $$\big|\mathrm{v}(D,C)\big|_{C,D \in C_{\mathcal{A}}^K} = \prod_{F \in F_{\mathcal{A}}^K \setminus C_{\mathcal{A}}^K} (1 - \mathrm{b}_F)^{\beta_F}.$$
\end{theorem}

\noindent Recall that a cone is an apartment of a central hyperplane arrangement. Aguiar and Mahajan computed that determinant for central hyperplane arrangements \cite[Theorem~8.11]{AgMa} and their cones \cite[Theorem~8.12]{AgMa}. And Gente computed that determinant for cone of hyperplane arrangements with use of the quantum bilinear form \cite[Theorem~4.5]{Ge}.

\begin{corollary} \label{MTh1}
Let $\mathcal{A}$ be a hyperplane arrangement in $\mathbb{R}^n$. Then, $$\big|\mathrm{v}(D,C)\big|_{C,D \in C_{\mathcal{A}}} = \prod_{F \in F_{\mathcal{A}} \setminus C_{\mathcal{A}}} (1 - \mathrm{b}_F)^{\beta_F}.$$
\end{corollary}

\begin{proof}
Consider the unique chamber $\mathbb{R}^n$ of the subset $\{\emptyset\}$ of $\mathcal{A}$.
\end{proof}

\noindent The original result of Varchenko \cite[Theorem~1.1]{Va2} is Corollary~\ref{MTh1} restricted to the quantum bilinear form. Pfeiffer and Randriamaro computed that determinant for Coxeter arrangements \cite[Theorem~1.1]{PfRa}. Furthermore, Corollary~\ref{MTh1} plays a key role in the computing of the Varchenko determinant of collages \cite[Theorem~1.6]{Ra3}.

\begin{example}
Consider the hyperplane arrangement $\mathcal{E} = \{H_1, H_2, H_3, H_4\}$ in Figure~\ref{ex}. Let for instance $F$ be the face such that $\epsilon_{\mathcal{E}}(F) = (-,-,0,0)$. Then, the centralization to $F$ is $\{H_3, H_4\}$, the weight of $F$ is $h_3^+ h_3^- h_4^+ h_4^-$, and the multiplicity of $F$ is $0$. Moreover, the Varchenko determinant for the apartment $H_1^-$ is
$$\begin{vmatrix}
1 & h_2^- & h_2^- h_3^- & h_4^- & h_2^- h_4^- & h_2^- h_3^- h_4^- \\
h_2^+ & 1 & h_3^- & h_2^+ h_4^- & h_4^- & h_3^- h_4^- \\
h_2^+ h_3^+ & h_3^+ & 1 & h_2^+ h_3^+ h_4^- & h_3^+ h_4^- & h_4^- \\
h_4^+ & h_2^- h_4^+ & h_2^- h_3^- h_4^+ & 1 & h_2^- & h_2^- h_3^- \\
h_2^+ h_4^+ & h_4^+ & h_3^- h_4^+ & h_2^+ & 1 & h_3^- \\
h_2^+ h_3^+ h_4^+ & h_3^+ h_4^+ & h_4^+ & h_2^+ h_3^+ & h_3^+ & 1
\end{vmatrix} = (1 - h_2^+ h_2^-)^2 \, (1 - h_3^+ h_3^-)^2 \, (1 - h_4^+ h_4^-)^3.$$

\begin{figure}[h]
	\centering
	\includegraphics[scale=0.8]{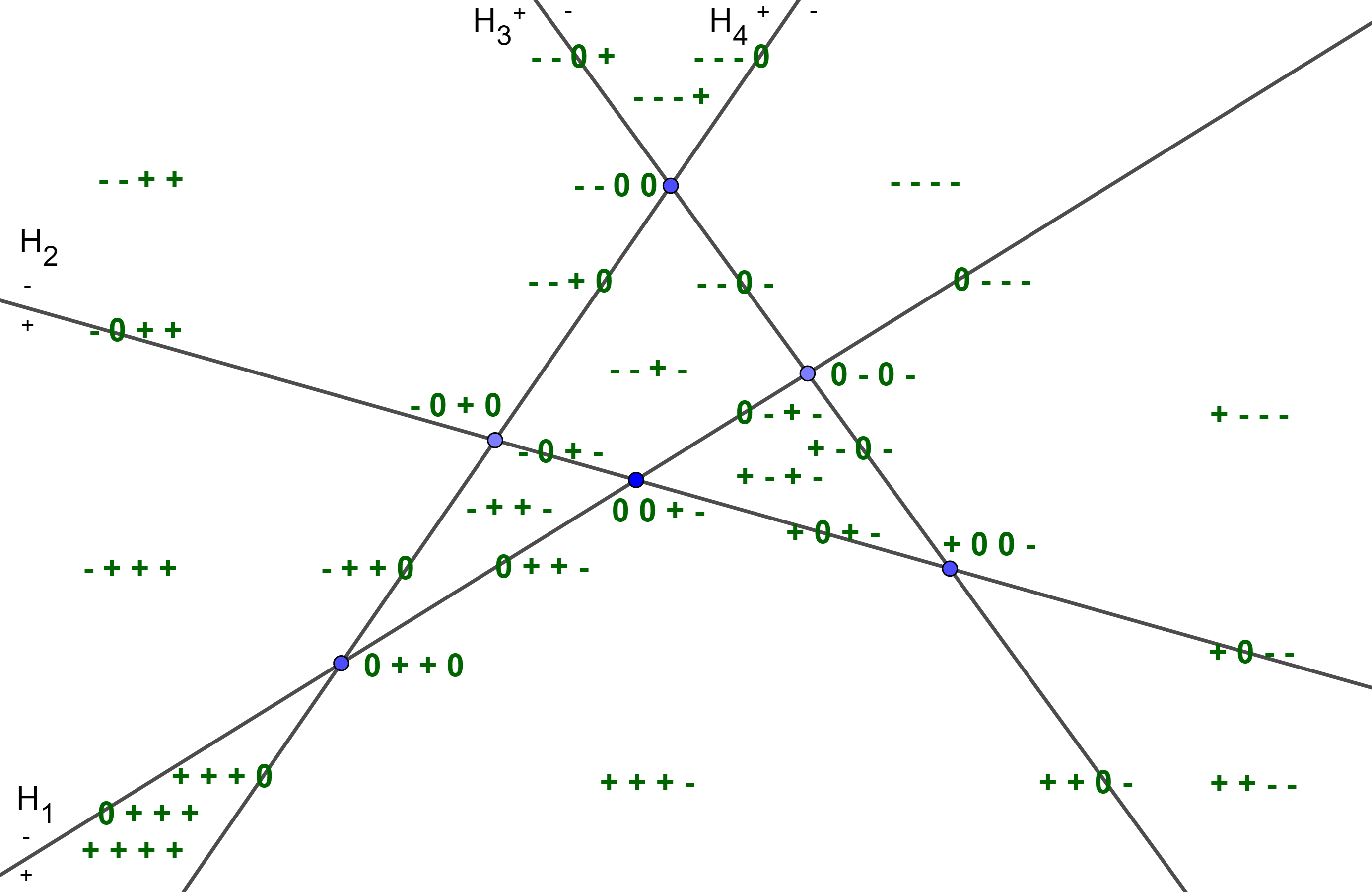}
	\caption{The Faces of $\mathcal{E} = \{H_1, H_2, H_3, H_4\}$}
	\label{ex}
\end{figure}
\end{example}

\noindent This article is structured as follows: We prove in Section~\ref{SecTi} that the set formed by the faces of hyperplane arrangements is a semigroup. In Section~\ref{SecEul}, we compute the Euler characteristics of CW complexes associated to hyperplane arrangements. From those results, we establish two generalizations of a Witt identity in Section~\ref{SecWit}. Finally, we use both generalizations to prove Theorem~\ref{MTh2} in Section~\ref{SecVar}.

\section{A Subsemigroup of Tits Monoids} \label{SecTi}

\noindent It is known that, for a central hyperplane arrangement $\mathcal{A}$, the set $F_{\mathcal{A}}$ is a monoid \cite[§ 1.4.2]{AgMa}, called Tits monoid, with the Tits product defined by: If $F,G \in F_{\mathcal{A}}$, then $FG$ is the face in $F_{\mathcal{A}}$ such that, for every $H \in \mathcal{A}$, $\displaystyle \epsilon_H(FG) := \begin{cases}
\epsilon_H(F) & \text{if}\ \epsilon_H(F) \neq 0, \\ \epsilon_H(G) & \text{otherwise}. \end{cases}$

\noindent We prove that $F_{\mathcal{A}}$ is a semigroup for the Tits product of a hyperplane arrangement $\mathcal{A}$.

\begin{proposition} \label{Tits}
Let $\mathcal{A}$ be a hyperplane arrangement in $\mathbb{R}^n$, and two faces $F,G \in F_{\mathcal{A}}$. Then, there exists a unique face $K \in F_{\mathcal{A}}$ such that, for every $H \in \mathcal{A}$,
$$\epsilon_H(K) = \begin{cases}
\epsilon_H(F) & \text{if}\ \epsilon_H(F) \neq 0, \\ \epsilon_H(G) & \text{otherwise}. \end{cases}$$
\end{proposition}

\begin{proof}
Define a path $\mathrm{p}:[0,1] \rightarrow \mathbb{R}^n$ from $x \in F$ to $y \in G$ by $\mathrm{p}(t) := (1-t)x + ty$. It is a homeomorphism between $[0,1]$ and $[x,y] := \mathrm{p}\big([0,1]\big)$. There exist $F_1, \dots, F_k \in F_{\mathcal{A}} \setminus \{F\}$, with $F_k = G$, such that
\begin{itemize}
\item $\forall i \in [k]:\ F_i \cap [x,y] \neq \emptyset$,
\item $\displaystyle [x,y] = (F \cap [x,y]) \sqcup \bigsqcup_{i \in [k]} (F_i \cap [x,y])$,
\item and $\forall i,j \in [k]:\ i<j \, \Rightarrow \, \sup \mathrm{p}^{-1}\big(F_i \cap [x,y]\big) < \inf \mathrm{p}^{-1}\big(F_j \cap [x,y]\big)$.
\end{itemize}

\noindent If $F_1 \prec F$, then $$\forall H \in \mathcal{A}:\ \epsilon_H(F) = 0 \ \Rightarrow \ \epsilon_H(G) = 0.$$ We deduce that $F$ is the face $K \in F_{\mathcal{A}}$ such that $\displaystyle \epsilon_H(K) = \begin{cases}
\epsilon_H(F) & \text{if}\ \epsilon_H(F) \neq 0, \\ \epsilon_H(G) & \text{otherwise} \end{cases}$.

\noindent Otherwise, $F_1 \succ F$. So, $F \subseteq \overline{F_1}$, which means $\forall H \in \mathcal{A}:\ \epsilon_H(F) \neq 0 \, \Rightarrow \, \epsilon_H(F_1) = \epsilon_H(F)$. Hence,
$$\forall H \in \mathcal{A},\, \epsilon_H(F) = 0:\ \epsilon_H(G) = \epsilon_H(F_1).$$ We deduce that $F_1$ is the face $K \in F_{\mathcal{A}}$ such that $\displaystyle \epsilon_H(K) = \begin{cases}
\epsilon_H(F) & \text{if}\ \epsilon_H(F) \neq 0, \\ \epsilon_H(G) & \text{otherwise} \end{cases}$.
\end{proof}

\begin{corollary}
Given a hyperplane arrangement $\mathcal{A}$ in $\mathbb{R}^n$, the set $F_{\mathcal{A}}$ together with the binary operation defined in Proposition~\ref{Tits} forms a semigroup.
\end{corollary}

\begin{proof}
It remains to prove the associativity of the binary operation. Let $E,F,G \in F_{\mathcal{A}}$, and $H \in \mathcal{A}$. Then,
$$\epsilon_H\big((EF)G\big) =\ \begin{cases}
\epsilon_H(E) & \text{if}\ \epsilon_H(E) \neq 0, \\
\epsilon_H(F) & \text{if}\ \epsilon_H(E)=0\ \text{and}\ \epsilon_H(F) \neq 0, \\
\epsilon_H(G) & \text{otherwise} \end{cases}\ = \epsilon_H\big(E(FG)\big).$$

\end{proof}

\noindent We could also deduce that $F_{\mathcal{A}}$ is a semigroup from the facts that every COM is a semigroup \cite[Proposition~2.12]{MaSaSt} and $F_{\mathcal{A}}$ is a COM \cite[Proposition~2.13]{MaSaSt}. However, we prefer to provide a proof in a hyperplane arrangement context especially because we will use the path $\mathrm{p}$ in the proof of Proposition~\ref{Tits} to prove Lemma~\ref{vCFD} in Section~\ref{SecVar}.

\section{The Euler Characteristic of Chambers}  \label{SecEul}

\noindent We compute the Euler characteristics of special subsets relative to hyperplane arrangements. Although the calculations are relatively simple, we mention the results as they are used later. 

\smallskip

\noindent The $m$-ball $B^m$ consists of a point if $m=0$, otherwise $B^m := \{x \in \mathbb{R}^m\ |\ \Vert x \Vert < 1\}$. An $m$-cell is a topological space homeomorphic to $B^m$. The dimension of an $m$-cell is defined to be $m$.

\noindent Let $X$ be a Hausdorff topological space, and assume that it is represented as a disjoint union of cells $e_{\alpha}$, $\alpha \in A$. Recall that the pair $\big(X, \{e_{\alpha}\}_{\alpha \in A}\big)$ is a CW complex if the following two conditions are satisfied:
\begin{enumerate}
\item If $\dim e_{\alpha} = m$, there exists a continuous map $f_{\alpha}: \overline{B^m} \rightarrow X$ such that
\begin{itemize}
\item the restriction of $f_{\alpha}$ to $B^m$ is a homeomorphism onto $e_{\alpha}$,
\item $f_{\alpha}(\partial B^m)$ is a union of finitely many cells of dimension less than $m$.
\end{itemize}
\item A subset $Y \subseteq X$ is closed in $X$ if and only if $Y \cap \overline{e_{\alpha}}$ is closed for any $\alpha$. 
\end{enumerate}

\noindent Recall that the Euler characteristic of a CW complex $\big(X, \{e_{\alpha}\}_{\alpha \in A}\big)$ is $$\chi(X) := \sum_{\alpha \in A} (-1)^{\dim e_{\alpha}}.$$

\noindent A polyhedron of a hyperplane arrangement $\mathcal{A}$ is a union of faces in $F_{\mathcal{A}}$ which is connected. We precisely investigate CW complexes $\big(P, \{F \in F_{\mathcal{A}}\ |\ F \subseteq P\}\big)$, where $P$ is a polyhedron $\mathcal{A}$.

\smallskip

\noindent \emph{Let $\breve{C}_{\mathcal{A}}$ be the set formed by the bounded chambers of a hyperplane arrangement $\mathcal{A}$. We distinguish three types of unbounded chambers $C \in C_{\mathcal{A}} \setminus \breve{C}_{\mathcal{A}}$ according to their frontiers:
\begin{enumerate}
\item if $\partial C$ is homeomorphic to $B^{n-1}$, we say that $C$ is of type $1$, and write $C \in C_{\mathcal{A}}^{(1)}$,
\item if $\partial C$ is homeomorphic to $\partial B^{n-1} \times \mathbb{R}$, we say that $C$ is of type $2$, and write $C \in C_{\mathcal{A}}^{(2)}$,
\item if $\partial C$ is homeomorphic to $\partial B^1 \times \mathbb{R}^{n-1}$, we say that $C$ is of type $3$, and write $C \in C_{\mathcal{A}}^{(3)}$.
\end{enumerate}}

\begin{example}
In Figure~\ref{ex2}, $C_1, C_2, C_3$ are chambers of type $2$, and the remaining chambers are of type $1$. Chambers of type $3$ are those whose frontiers are two parallel hyperplanes.

\begin{figure}[h]
	\centering
	\includegraphics[scale=0.3]{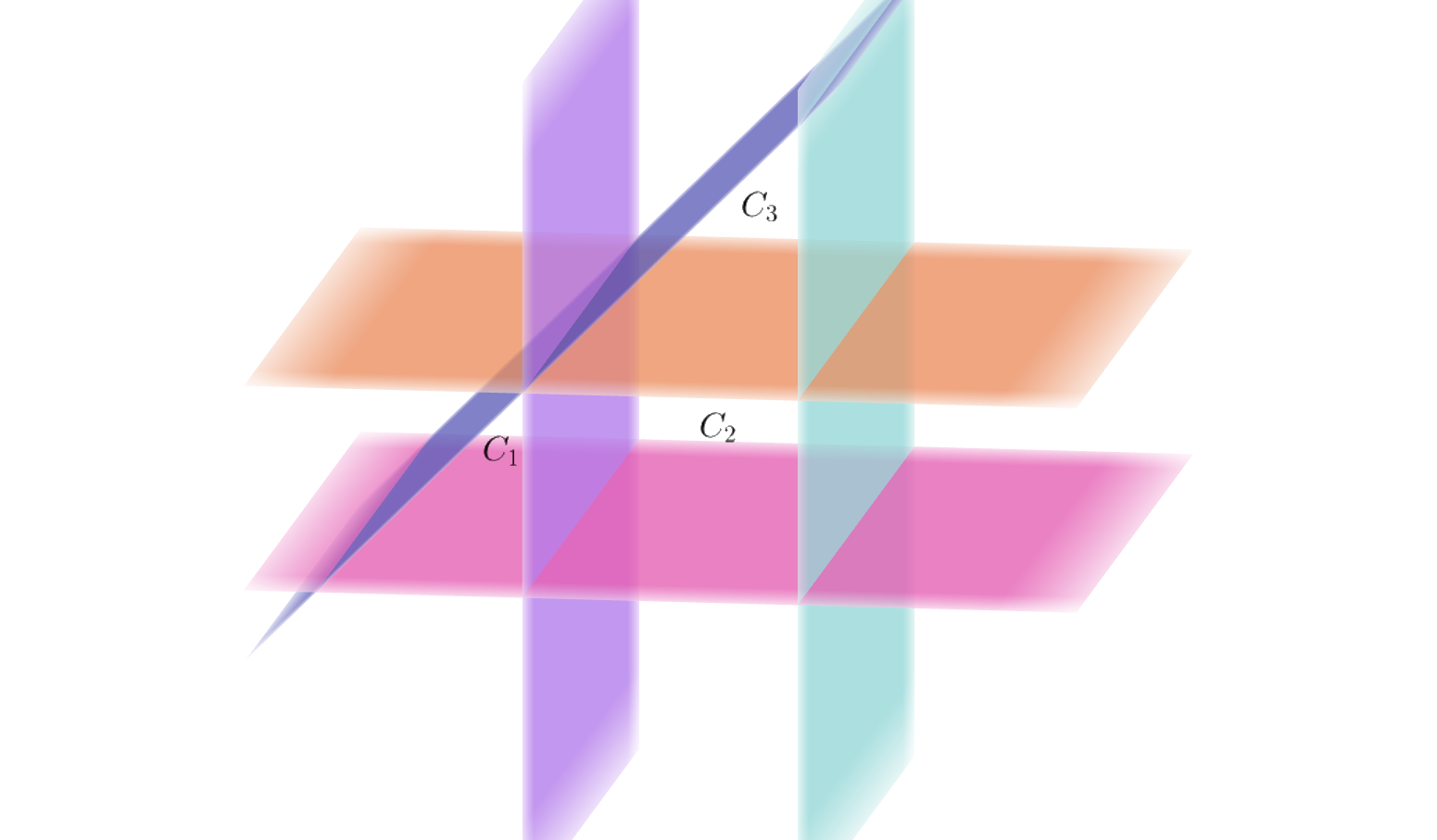}
	\caption{Chambers of Types $1$ and $2$}
	\label{ex2}
\end{figure}
\end{example}

\begin{lemma} \label{Ch}
Let $\mathcal{A}$ be a hyperplane arrangement in $\mathbb{R}^n$, and $C \in C_{\mathcal{A}}$. Then,
$$\chi(\overline{C}) = \begin{cases} 1 & \text{if} \quad C \in \breve{C}_{\mathcal{A}}, \\
0 & \text{if}\quad C \in C_{\mathcal{A}}^{(1)}, \\
-1 & \text{if} \quad C \in C_{\mathcal{A}}^{(2)}, \\
(-1)^{n-1} & \text{if} \quad C \in C_{\mathcal{A}}^{(3)}. \end{cases}$$
\end{lemma}

\begin{proof}
If $C \in \breve{C}_{\mathcal{A}}$, then $\overline{C}$ is homotopy equivalent to a point, and $\chi(\overline{C}) = 1$.

\noindent (1) If $C \in C_{\mathcal{A}}^{(1)}$, $\chi(\overline{C}) = \chi(C) + \chi(\partial C) = \chi(B^n) + \chi(B^{n-1}) = (-1)^n + (-1)^{n-1} = 0$. 

\noindent (2) If $C \in C_{\mathcal{A}}^{(2)}$, divide $C$ in two chambers $C_1$ and $C_2$ of type $1$ by adding a bounded $(n-1)$-dimensional cell $P$ in $C$. Then, $$\chi(\overline{C}) = \chi(\overline{C_1}) + \chi(\overline{C_2}) - \chi(\overline{C_1} \cap \overline{C_2}) = - \chi(\overline{C_1} \cap \overline{C_2}) = - \chi(\overline{P}) = -1.$$

\noindent (3) If $C \in C_{\mathcal{A}}^{(3)}$, $\chi(\overline{C}) = \chi(C) + \chi(\partial C) = \chi(B^n) + 2\chi(B^{n-1}) = (-1)^{n-1}$.
\end{proof}

\noindent A panel of a chamber $C \in C_{\mathcal{A}}$ is a face $F \in F_{\mathcal{A}} \setminus C_{\mathcal{A}}$ such that $F \subseteq \overline{C}$ and $\dim F = n-1$. Denote the set formed by the panels of $C$ by $P_C$.

\begin{lemma} \label{ChM}
Let $\mathcal{A}$ be a hyperplane arrangement in $\mathbb{R}^n$, $C \in C_{\mathcal{A}}$, and $\{F_j\}_{j \in J}$ a nonempty set formed by panels of $C$ such that $\{F_j\}_{j \in J} \neq P_C$, and, if $\#J > 1$,
$$\forall j \in J,\, \exists k \in J \setminus \{j\}:\ \dim \overline{F_j} \cap \overline{F_k} = n-2.$$ 
Then, $$\chi\Big(\overline{C} \setminus \bigcup_{j \in J} \overline{F_j} \Big) = \begin{cases}
-1 & \text{if} \ C \in C_{\mathcal{A}}^{(1)}\ \text{and $\displaystyle \bigcup_{j \in J} \overline{F_j}$ is bounded}, \\ 0 & \text{otherwise}.\end{cases}$$
\end{lemma}

\begin{proof}
If $C \in \breve{C}_{\mathcal{A}}$, then $\displaystyle \bigcup_{j \in J} \overline{F_j}$ is homotopy equivalent to a point.

\noindent (1) Suppose that $C \in C_{\mathcal{A}}^{(1)}$.
\begin{itemize}
\item If $\displaystyle \bigcup_{j \in J} \overline{F_j}$ is bounded, then it is homotopy equivalent to a point.
\item Otherwise, $\displaystyle \mathrm{int}\Big(\bigcup_{j \in J} \overline{F_j}\Big)$ is homeomorphic to $B^{n-1}$ while $\displaystyle \partial \bigcup_{j \in J} \overline{F_j}$ to $B^{n-2}$.
\end{itemize}

\noindent (2) If $C \in C_{\mathcal{A}}^{(2)}$, then $\displaystyle \overline{C} \setminus \bigcup_{j \in J} \overline{F_j}$ is homeomorphic to the closure of a chamber of type $1$.

\noindent (3) If $C \in C_{\mathcal{A}}^{(3)}$, there is a hyperplane $H$ such that $\displaystyle \overline{C} \setminus \bigcup_{j \in J} \overline{F_j} = C \sqcup H$. 
\end{proof}

\section{Two Generalizations of a Witt Identity}  \label{SecWit}

\noindent We extend a Witt identity \cite[Proposition~7.30]{AgMa} to affine hyperplane arrangements.

\smallskip

\noindent A nested face of a hyperplane arrangement $\mathcal{A}$ is a pair $(F,G)$ of faces in $F_{\mathcal{A}}$ such that $F \prec G$. For a nested face $(F,G)$, define the set of faces $F_{\mathcal{A}}^{(F,G)} := \{K \in F_{\mathcal{A}}\ |\ F \preceq K \preceq G\}$.

\begin{lemma}
Let $\mathcal{A}$ be a hyperplane arrangement in $\mathbb{R}^n$, $D \in C_{\mathcal{A}}$, and $(A,D)$ a nested face. Then, $C_{\mathcal{A}}$ has a chamber $\tilde{D}_A$ whose sign sequence is defined, for every $H \in \mathcal{A}$, by
$$\epsilon_H(\tilde{D}_A) = \begin{cases} -\epsilon_H(D) & \text{if}\ \epsilon_H(A) = 0, \\
\epsilon_H(A) & \text{otherwise}. \end{cases}$$
\end{lemma}

\begin{proof}
Take two points $x \in D$, and $y \in A$. From $x$, the ray $\big\{x(1-t) + yt\ \big|\ t \in \mathbb{R}_{\geq 0}\big\}$ successively meets the chamber $D$, the face $A$, and another-first chamber $C$:
\begin{itemize}
\item if $H \in \mathcal{A}_A$, then $\epsilon_H(C) = -\epsilon_H(D)$,
\item else, $\epsilon_H(C) = \epsilon_H(D) = \epsilon_H(A)$.
\end{itemize}
\end{proof}

\noindent Let $\mathcal{A}$ be a hyperplane arrangement, and $\mathrm{c}_{\mathcal{A}} := \min \{\dim F\ |\ F \in F_{\mathcal{A}}\}$. The rank of a face $F \in F_{\mathcal{A}}$ is $\mathrm{rk}\,F := \dim F - \mathrm{c}_{\mathcal{A}}$. Assign a variable $x_C$ to each chamber $C \in C_{\mathcal{A}}$.

\begin{proposition}  \label{Witt}
Let $\mathcal{A}$ be a hyperplane arrangement in $\mathbb{R}^n$, $D \in C_{\mathcal{A}}$, and $(A,D)$ a nested face. Then,
$$\sum_{F \in F_{\mathcal{A}}^{(A,D)}} (-1)^{\mathrm{rk}\,F} \sum_{\substack{C \in C_{\mathcal{A}} \\ FC = D}} x_C \, = \, (-1)^{\mathrm{rk}\,D} \sum_{\substack{C \in C_{\mathcal{A}} \\ AC = \tilde{D}_A}} x_C.$$
\end{proposition}

\begin{proof}
We have $\displaystyle \sum_{F \in F_{\mathcal{A}}^{(A,D)}} (-1)^{\mathrm{rk}\,F} \sum_{\substack{C \in C_{\mathcal{A}} \\ FC = D}} x_C = \sum_{C \in C_{\mathcal{A}}} \Big( \sum_{\substack{F \in F_{\mathcal{A}}^{(A,D)} \\ FC = D}} (-1)^{\mathrm{rk}\,F} \Big) x_C$. Note that
$$F_{\mathcal{A}}^{(A,D)} = \big\{F \in F_{\mathcal{A}}\ |\ F \preceq D,\, \epsilon_{\mathcal{A} \setminus \mathcal{A}_A}(F) = \epsilon_{\mathcal{A} \setminus \mathcal{A}_A}(A)\big\}.$$
\begin{itemize}
\item If $\epsilon_{\mathcal{A}_A}(C) = \epsilon_{\mathcal{A}_A}(\tilde{D}_A)$, then $\displaystyle \sum_{\substack{F \in F_{\mathcal{A}}^{(A,D)} \\ FC = D}} (-1)^{\mathrm{rk}\,F} = (-1)^{\mathrm{rk}\,D}$.
\end{itemize}

\noindent Define the bijection $b: \{F \in F_{\mathcal{A}}\ |\ A \preceq F\} \rightarrow F_{\mathcal{A}_A}$ such that $\epsilon_{\mathcal{A}_A}\big(b(F)\big) = \epsilon_{\mathcal{A}_A}(F)$.

\begin{itemize}
\item If $\epsilon_{\mathcal{A}_A}(C) = \epsilon_{\mathcal{A}_A}(D)$, then
\begin{align*}
\sum_{\substack{F \in F_{\mathcal{A}}^{(A,D)} \\ FC = D}} (-1)^{\mathrm{rk}\,F} & = (-1)^{- \mathrm{c}_{\mathcal{A}}} \sum_{F \in F_{\mathcal{A}}^{(A,D)}} (-1)^{\dim F} \\
& = (-1)^{\mathrm{c}_{\mathcal{A}}} \sum_{F \in b\big(F_{\mathcal{A}}^{(A,D)}\big)} (-1)^{\dim F} \\
& = (-1)^{\mathrm{c}_{\mathcal{A}}} \chi\Big(\overline{b(D)}\Big) \\
& = 0 \quad \text{from Lemma \ref{Ch} as $b(D)$ is a chamber of type $1$}.
\end{align*}
\item The case $\epsilon_{\mathcal{A}_A}(C) \notin \big\{\epsilon_{\mathcal{A}_A}(D), \epsilon_{\mathcal{A}_A}(\tilde{D}_A)\big\}$ remains. That condition imposes $\#\mathcal{A}_A > 1$. Assume that, for every $H \in \mathcal{A}_A$, $\epsilon_{H}(D) = +$, and define the hyperplane arrangement $\mathcal{A}_A^C := \{H \in \mathcal{A}_A\ |\ \epsilon_{H}(C) = -\}$. Remark that if $\#\mathcal{A}_A^C > 1$, then $$\forall H \in \mathcal{A}_A^C,\, \exists H' \in \mathcal{A}_A^C \setminus \{H\}:\ \dim \overline{H} \cap \overline{H'} = n-2.$$
We obtain,
\begin{align*}
\sum_{\substack{F \in F_{\mathcal{A}}^{(A,D)} \\ FC = D}} (-1)^{\mathrm{rk}\,F} & = (-1)^{- \mathrm{c}_{\mathcal{A}}} \sum_{\substack{F \in F_{\mathcal{A}}^{(A,D)} \\ \forall H \in \mathcal{A}_A^C:\ \epsilon_H(F) = +}} (-1)^{\dim F} \\
& = (-1)^{\mathrm{c}_{\mathcal{A}}} \sum_{\substack{F \in b\big(F_{\mathcal{A}}^{(A,D)}\big) \\ \forall H \in \mathcal{A}_A^C:\ \epsilon_H(F) = +}} (-1)^{\dim F} \\
& = (-1)^{\mathrm{c}_{\mathcal{A}}} \chi\Big(\overline{b(D)} \setminus \bigcup_{H \in \mathcal{A}_A^C} \overline{b(D)} \cap H\Big) \\
& = 0 \quad \text{using Lemma \ref{ChM} with $\displaystyle \bigcup_{H \in \mathcal{A}_A^C} \overline{b(D)} \cap H$ unbounded}.
\end{align*}
\end{itemize}
So $\displaystyle \sum_{F \in F_{\mathcal{A}}^{(A,D)}} (-1)^{\mathrm{rk}\,F} \sum_{\substack{C \in C_{\mathcal{A}} \\ FC = D}} x_C \, = \, (-1)^{\mathrm{rk}\,D} \sum_{\substack{C \in C_{\mathcal{A}} \\ \epsilon_{\mathcal{A}_A}(C) \, = \, \epsilon_{\mathcal{A}_A}(\tilde{D}_A)}} x_C \, = \, (-1)^{\mathrm{rk}\,D} \sum_{\substack{C \in C_{\mathcal{A}} \\ AC = \tilde{D}_A}} x_C$.
\end{proof}

\noindent The set of faces composing the closure of a chamber $D \in C_{\mathcal{A}}$ is $F_{\overline{D}} := \{F \in F_{\mathcal{A}}\ |\ F \preceq D\}$.

\begin{proposition}   \label{Witt2}
Let $\mathcal{A}$ be a hyperplane arrangement in $\mathbb{R}^n$, and $D \in \breve{C}_{\mathcal{A}} \cup C_{\mathcal{A}}^{(2)} \cup C_{\mathcal{A}}^{(3)}$. Then,
$$\sum_{F \in F_{\overline{D}}} (-1)^{\mathrm{rk}\,F} \sum_{\substack{C \in C_{\mathcal{A}} \\ FC = D}} x_C \, = \, \begin{cases}
(-1)^{\mathrm{c}_{\mathcal{A}}} x_D & \text{if} \quad D \in \breve{C}_{\mathcal{A}}, \\
(-1)^{1 + \mathrm{c}_{\mathcal{A}}} x_D & \text{if} \quad D \in C_{\mathcal{A}}^{(2)}, \\
(-1)^{\mathrm{c}_{\mathcal{A}} + n-1} x_D & \text{if} \quad D \in C_{\mathcal{A}}^{(3)}. \end{cases}$$
\end{proposition}

\begin{proof}
We have $\displaystyle \sum_{F \in F_{\overline{D}}} (-1)^{\mathrm{rk}\,F} \sum_{\substack{C \in C_{\mathcal{A}} \\ FC = D}} x_C = \sum_{C \in C_{\mathcal{A}}} \Big( \sum_{\substack{F \in F_{\overline{D}} \\ FC = D}} (-1)^{\mathrm{rk}\,F} \Big) x_C.$ If $C \neq D$, define the hyperplane arrangement $\mathcal{A}_{C,D} := \{H \in \mathcal{A}\ |\ \epsilon_{H}(C) \neq \epsilon_{H}(D)\}$. Remark that if $\#\mathcal{A}_{C,D} > 1$, then $$\forall H \in \mathcal{A}_{C,D},\, \exists H' \in \mathcal{A}_{C,D} \setminus \{H\}:\ \dim \overline{H} \cap \overline{H'} = n-2.$$ We obtain
\begin{align*}
\sum_{\substack{F \in F_{\overline{D}} \\ FC = D}} (-1)^{\mathrm{rk}\,F} & = (-1)^{\mathrm{c}_{\mathcal{A}}} \sum_{\substack{F \in F_{\overline{D}} \\ \forall H \in \mathcal{A}_{C,D}:\ \epsilon_H(F) \neq 0}} (-1)^{\dim F} \\
& = (-1)^{\mathrm{c}_{\mathcal{A}}} \chi\Big(\overline{D} \setminus \bigcup_{H \in \mathcal{A}_{C,D}} \overline{D} \cap H\Big) \\
& = 0 \quad \text{from Lemma \ref{ChM}}.
\end{align*}
If $C = D$, then $\displaystyle \sum_{\substack{F \in F_{\overline{D}} \\ FC = D}} (-1)^{\mathrm{rk}\,F} = (-1)^{\mathrm{c}_{\mathcal{A}}} \chi(\overline{D}) = \begin{cases}
(-1)^{\mathrm{c}_{\mathcal{A}}} & \text{if} \quad D \in \breve{C}_{\mathcal{A}}, \\
(-1)^{1 + \mathrm{c}_{\mathcal{A}}} & \text{if} \quad D \in C_{\mathcal{A}}^{(2)}, \\
(-1)^{\mathrm{c}_{\mathcal{A}} + n-1} & \text{if} \quad D \in C_{\mathcal{A}}^{(3)}. \end{cases}$
\end{proof}

\section{The Varchenko Matrix for Apartments}  \label{SecVar}

\noindent We prove the main result in this section. For that, we principally use the extensions to affine hyperplane arrangements of the maps $\gamma_{\mathcal{A}}$ and $\mathrm{m}$ defined by Aguiar and Mahajan \cite[§ 8.4]{AgMa}.

\begin{lemma} \label{vCFD}
Let $\mathcal{A}$ be a hyperplane arrangement in $\mathbb{R}^n$, $C, D \in C_{\mathcal{A}}$, and $F \preceq C$. Then, $$\mathrm{v}(C,D) = \mathrm{v}(C,FD) \, \mathrm{v}(FD,D).$$
\end{lemma}

\begin{proof} From the proof of Proposition \ref{Tits}, we know that, if $\mathrm{p}(t)$ is the directed line-segment from a point $x \in F$ to a point $y \in D$, then $FD$ is the first chamber $F_1$ that $\mathrm{p}(t)$ meets. 
\begin{itemize}
\item If $F_1 = C$, then $\mathscr{H}(C,F_1) \sqcup \mathscr{H}(F_1,D) = \mathscr{H}(C,C) \sqcup \mathscr{H}(C,D) = \mathscr{H}(C,D)$.
\item Else, if $\mathcal{A}' = \{H \in \mathcal{A}\ |\ \overline{C} \cap \overline{F_1} \subseteq H\}$, then $$\mathscr{H}(C,F_1) \sqcup \mathscr{H}(F_1,D) = \{H^{\epsilon_H(C)}\ |\ H \in \mathcal{A}'\} \sqcup \big\{H^{\epsilon_H(F_1)}\ \big|\ H \in \mathcal{A} \setminus \mathcal{A}',\, \epsilon_H(F_1) \neq \epsilon_H(D)\big\} = \mathscr{H}(C,D).$$
\end{itemize}
\end{proof}

\noindent The module of $R_{\mathcal{A}}$--linear combinations of chambers in $C_{\mathcal{A}}$ is
$\displaystyle M_{\mathcal{A}} := \Big\{\sum_{C \in C_{\mathcal{A}}} x_C C\ \Big|\ x_C \in R_{\mathcal{A}}\Big\}$.

\noindent Let $\{C^*\}_{C \in C_{\mathcal{A}}}$ be the dual basis of the basis $C_{\mathcal{A}}$ of $M_{\mathcal{A}}$. Define the linear map $\gamma_{\mathcal{A}}: M_{\mathcal{A}} \rightarrow M_{\mathcal{A}}^*$, for a chamber $D \in C_{\mathcal{A}}$, by
$$\gamma_{\mathcal{A}}(D) := \sum_{C \in C_{\mathcal{A}}} \mathrm{v}(D,C)\, C^*.$$

\noindent For a nested face $(A,D)$, where $D$ is a chamber, let $\displaystyle \mathrm{m}(A,D) := \sum_{\substack{C \in C_{\mathcal{A}} \\ AC = D}} \mathrm{v}(D,C)\, C^* \in M_{\mathcal{A}}^*$.

\noindent Define the extension ring $\displaystyle B_{\mathcal{A}} := \bigg\{\frac{p}{\displaystyle \prod_{F \in F_{\mathcal{A}}  \setminus C_{\mathcal{A}}} (1 - \mathrm{b}_F)^{k_F}}\ \bigg|\ p \in R_{\mathcal{A}},\, k_F \in \mathbb{N}\bigg\}$ of $R_{\mathcal{A}}$ from the weights of the faces in $F_{\mathcal{A}}$.

\begin{proposition} \label{mAD}
Let $\mathcal{A}$ be a hyperplane arrangement in $\mathbb{R}^n$, $D \in C_{\mathcal{A}}$, and $(A,D)$ a nested face. Then, $$\mathrm{m}(A,D) = \sum_{C \in C_{\mathcal{A}}} x_C \, \gamma_{\mathcal{A}}(C) \quad \text{with} \quad x_C \in B_{\mathcal{A}}.$$
\end{proposition}

\begin{proof}
The backward induction proof is inspired by a part of the proof of \cite[Proposition~8.13]{AgMa}. We obviously have $\mathrm{m}(D,D) = \gamma_{\mathcal{A}}(D)$. The generalized Witt identity in Proposition \ref{Witt} applied to $x_C = \mathrm{v}(D,C)\, C^*$ in addition to Lemma \ref{vCFD} yield
$$\sum_{F \in F_{\mathcal{A}}^{(A,D)}} (-1)^{\mathrm{rk}\,F} \mathrm{m}(F,D) = (-1)^{\mathrm{rk}\,D} \sum_{\substack{C \in C_{\mathcal{A}} \\ AC = \tilde{D}_A}} \mathrm{v}(D,C)\, C^* = (-1)^{\mathrm{rk}\,D} \, \mathrm{v}(D,\tilde{D}_A) \, \mathrm{m}(A,\tilde{D}_A).$$
Hence, $\displaystyle \mathrm{m}(A,D) - (-1)^{\mathrm{rk}\,D - \mathrm{rk}\,A} \, \mathrm{v}(D,\tilde{D}_A) \, \mathrm{m}(A,\tilde{D}_A) = \sum_{F \in F_{\mathcal{A}}^{(A,D)} \setminus \{A\}} (-1)^{\mathrm{rk}\,F - \mathrm{rk}\,A +1} \mathrm{m}(F,D)$.
By induction hypothesis, for every $C \in C_{\mathcal{A}}$, there exists $a_C \in B_{\mathcal{A}}$, such that $$\sum_{F \in F_{\mathcal{A}}^{(A,D)} \setminus \{A\}} (-1)^{\mathrm{rk}\,F - \mathrm{rk}\,A +1} \mathrm{m}(F,D) = \sum_{C \in C_{\mathcal{A}}} a_C \, \gamma_{\mathcal{A}}(C).$$
Since $A \preceq \tilde{D}_A$ and $A\widetilde{\tilde{D}_A} = D$, by replacing $D$ with $\tilde{D}_A$, there exists also $e_C \in B_{\mathcal{A}}$ for every $C \in C_{\mathcal{A}}$ such that
$$\displaystyle \mathrm{m}(A, \tilde{D}_A) - (-1)^{\mathrm{rk}\,\tilde{D}_A - \mathrm{rk}\,A} \, \mathrm{v}(\tilde{D}_A, D) \, \mathrm{m}(A,D) = \sum_{C \in C_{\mathcal{A}}} e_C \, \gamma_{\mathcal{A}}(C).$$
Therefore, $\displaystyle \mathrm{m}(A,D) = \sum_{C \in C_{\mathcal{A}}} \frac{a_C + (-1)^{\mathrm{rk}\,D - \mathrm{rk}\,A} \, \mathrm{v}(D, \tilde{D}_A) \, e_C}{1 - \mathrm{b}_A} \gamma_{\mathcal{A}}(C)$.
\end{proof}

\begin{theorem} \label{D*}
Let $\mathcal{A}$ be a hyperplane arrangement in $\mathbb{R}^n$, and $D \in C_{\mathcal{A}}$. Then,
$$D^* = \sum_{C \in C_{\mathcal{A}}} x_C \, \gamma_{\mathcal{A}}(C) \quad \text{with} \quad x_C \in B_{\mathcal{A}}.$$
\end{theorem}

\begin{proof}
Suppose first that $D \in \breve{C}_{\mathcal{A}} \cup C_{\mathcal{A}}^{(2)} \cup C_{\mathcal{A}}^{(3)}$. Applying $x_C = \mathrm{v}(D,C)\, C^*$ to Proposition~\ref{Witt2}, we obtain
$$\sum_{F \in F_{\overline{D}}} (-1)^{\mathrm{rk}\,F} \mathrm{m}(F,D) = \lambda D^* \quad \text{with} \quad \lambda \in \{-1,1\}.$$
From Proposition \ref{mAD}, we conclude that $\displaystyle D^* = \sum_{C \in C_{\mathcal{A}}} x_C \, \gamma_{\mathcal{A}}(C)$ with $x_C \in B_{\mathcal{A}}$.

\noindent Suppose now that $D \in C_{\mathcal{A}}^{(1)}$. Consider the hyperplane arrangement $\mathcal{A}' = \mathcal{A} \sqcup \{H'\}$ such that
\begin{itemize}
\item $H'$ divides $D$ into two chambers $D_b \in \breve{C}_{\mathcal{A}'} \cup C_{\mathcal{A}'}^{(2)} \cup C_{\mathcal{A}'}^{(3)}$ and $D_u \in C_{\mathcal{A}'}^{(1)}$,
\item for every $C \in \breve{C}_{\mathcal{A}} \cup C_{\mathcal{A}}^{(2)} \cup C_{\mathcal{A}}^{(3)}$, $H' \cap C = \emptyset$,
\item if $C_{\mathcal{A}}'$ is the set of chambers in $C_{\mathcal{A}}$ cut by $H'$, then, for every $C \in C_{\mathcal{A}} \setminus C_{\mathcal{A}}'$, $C \subset H'^+$,
\item if we denote by $C_b$ and $C_u$ the two chambers obtained from the cut of $C \in C_{\mathcal{A}}'$ by $H'$, then we assume that $C_b \subset H'^+$.
\end{itemize}
Then,
\begin{align*}
& D_b^* = \sum_{C \in C_{\mathcal{A}'}} x_C \, \gamma_{\mathcal{A}'}(C) \quad \text{with} \quad x_C \in B_{\mathcal{A}'} \\
& D_b^* = \sum_{C \in C_{\mathcal{A}} \setminus C_{\mathcal{A}}'} x_C \, \gamma_{\mathcal{A}'}(C) + \sum_{C' \in C_{\mathcal{A}}'} \Big( x_{C_b'} \, \gamma_{\mathcal{A}'}(C_b') + x_{C_u'} \, \gamma_{\mathcal{A}'}(C_u') \Big) \\
& D_b^* - \sum_{C \in C_{\mathcal{A}} \setminus C_{\mathcal{A}}'} x_C \, \gamma_{\mathcal{A}'}(C) - \sum_{C' \in C_{\mathcal{A}}'} x_{C_b'} \, \gamma_{\mathcal{A}'}(C_b') = \sum_{C' \in C_{\mathcal{A}}'} x_{C_u'} \, \gamma_{\mathcal{A}'}(C_u').
\end{align*}
Setting $h_{H'}^+ = h_{H'}^- = 0$, we obtain
$$\displaystyle D_b^* - \sum_{C \in C_{\mathcal{A}} \setminus C_{\mathcal{A}}'} x_C \, \gamma_{\mathcal{A}'}(C) - \sum_{C' \in C_{\mathcal{A}}'} x_{C_b'} \, \gamma_{\mathcal{A}'}(C_b')\, \in \, \Big\{\sum_{C \in C_{\mathcal{A}} \setminus C_{\mathcal{A}}'} x_{C^*} C^* + \sum_{C' \in C_{\mathcal{A}}'} x_{C_b{'}^*} C_b{'}^*\ \Big|\ x_{C^*}, x_{C_b{'}^*} \in B_{\mathcal{A}}\Big\},$$
and $\displaystyle \sum_{C' \in C_{\mathcal{A}}'} x_{C_u'} \, \gamma_{\mathcal{A}'}(C_u')\, \in \, \Big\{\sum_{C' \in C_{\mathcal{A}}'} x_{C_u{'}^*} C_u{'}^*\ \Big|\ x_{C_u{'}^*} \in B_{\mathcal{A}}\Big\}$. The only possibility is
$$D_b^* - \sum_{C \in C_{\mathcal{A}} \setminus C_{\mathcal{A}}'} x_C \, \gamma_{\mathcal{A}'}(C) - \sum_{C' \in C_{\mathcal{A}}'} x_{C_b'} \, \gamma_{\mathcal{A}'}(C_b')\ =\ \sum_{C' \in C_{\mathcal{A}}'} x_{C_u'} \, \gamma_{\mathcal{A}'}(C_u')\ =\ 0.$$
Finally, replacing $C_b'$ by $C'$ for every $C' \in C_{\mathcal{A}}'$, we conclude that
$$D_b^* = \sum_{C \in C_{\mathcal{A}} \setminus C_{\mathcal{A}}'} x_C \, \gamma_{\mathcal{A}}(C) + \sum_{C' \in C_{\mathcal{A}}'} x_{C'} \, \gamma_{\mathcal{A}}(C') = D^*.$$
\end{proof}

\begin{proposition} \label{fdetV}
Let $\mathcal{A}$ be a hyperplane arrangement in $\mathbb{R}^n$. For every face $F \in F_{\mathcal{A}} \setminus C_{\mathcal{A}}$, there is a nonnegative integer $l_F$ such that the determinant of the Varchenko matrix of $\mathcal{A}$ has the from $$\det V_{\mathcal{A}} = \prod_{F \in F_{\mathcal{A}} \setminus C_{\mathcal{A}}} (1 - \mathrm{b}_F)^{l_F}.$$ 
\end{proposition}

\begin{proof}
It is clear that $V_{\mathcal{A}}$ is the matrix representation of $\gamma_{\mathcal{A}}$. The determinant of $V_{\mathcal{A}}$ is a polynomial in $R_{\mathcal{A}}$ with constant term $1$, so $V_{\mathcal{A}}$ is invertible. From Theorem \ref{D*}, we know that, for every chamber $D \in C_{\mathcal{A}}$, there exist $x_C \in B_{\mathcal{A}}$ such that $\displaystyle D^* = \sum_{C \in C_{\mathcal{A}}} x_C \gamma_{\mathcal{A}}(C)$. Hence, $$\gamma_{\mathcal{A}}^{-1}(D^*) = \sum_{C \in C_{\mathcal{A}}} x_C \, C \quad \text{with} \quad x_C \in B_{\mathcal{A}}.$$
As the matrix representation of $\gamma_{\mathcal{A}}^{-1}$ is $V_{\mathcal{A}}^{-1}$, each entry of $V_{\mathcal{A}}^{-1}$ is then an element of $B_{\mathcal{A}}$. To finish, note that $\displaystyle V_{\mathcal{A}}^{-1} = \frac{\mathrm{adj}(V_{\mathcal{A}})}{\det V_{\mathcal{A}}}$, where each entry of $\mathrm{adj}(V_{\mathcal{A}})$ is a polynomial in $R_{\mathcal{A}}$. Then, the only possibility is $\det V_{\mathcal{A}}$ has the form $\displaystyle k \prod_{F \in F_{\mathcal{A}} \setminus C_{\mathcal{A}}} (1 - \mathrm{b}_F)^{l_F}$, with $k \in \mathbb{Z}$. As the constant term of $\det V_{\mathcal{A}}$ is $1$, we deduce that $k=1$.
\end{proof}

\noindent Define the subring $\displaystyle B_{\mathcal{A}}^K := \bigg\{\frac{p}{\displaystyle \prod_{F \in F_{\mathcal{A}}^K \setminus C_{\mathcal{A}}^K} (1 - \mathrm{b}_F)^{k_F}}\ \bigg|\ p \in R_{\mathcal{A}},\, k_F \in \mathbb{N}\bigg\}$ of $B_{\mathcal{A}}$ from the weights of the faces in $F_{\mathcal{A}}^K$. 

\begin{definition}
The \textbf{Varchenko matrix} for an apartment $K$ of a hyperplane arrangement $\mathcal{A}$ in $\mathbb{R}^n$ is
$$\displaystyle V_{\mathcal{A}}^K := \big(\mathrm{v}(D,C)\big)_{C,D \in C_{\mathcal{A}}^K}.$$
\end{definition}

\begin{proposition} \label{PdetK}
Let $\mathcal{A}$ a hyperplane arrangement in $\mathbb{R}^n$, and $K \in K_{\mathcal{A}}$. Then, for every face $F \in F_{\mathcal{A}}^K \setminus C_{\mathcal{A}}^K$, there exists a nonnegative integer $l_F$ such that $$\det V_{\mathcal{A}}^K = \prod_{F \in F_{\mathcal{A}}^K \setminus C_{\mathcal{A}}^K} (1 - \mathrm{b}_F)^{l_F}.$$ 
\end{proposition}

\begin{proof}
Let $\mathcal{K}$ be the subset of $\mathcal{A}$ containing the hyperplanes $H$ such that $\dim H \cap \overline{K} = n-1$. For every $H \in \mathcal{K}$, set $h_H^+ = h_H^- = 0$. Hence, $\mathrm{v}(D,C) = 0$ whenever one of $C$ or $D$ is a chamber in $C_{\mathcal{A}}^K$ but the other not. For every $D \in C_{\mathcal{A}}^K$, and any nested face $(A,D)$, define $$\gamma_{\mathcal{A}}^K(D) := \sum_{C \in C_{\mathcal{A}}^K} \mathrm{v}(D,C)\, C^* \quad \text{and} \quad \mathrm{m}_K(A,D) := \sum_{\substack{C \in C_{\mathcal{A}}^K \\ AC = D}} \mathrm{v}(D,C)\, C^*.$$
In that case, $\gamma_{\mathcal{A}}(D) = \gamma_{\mathcal{A}}^K(D)$ and $\mathrm{m}(A,D) = \mathrm{m}_K(A,D)$. With a backward induction similar to the proof of Proposition \ref{mAD}, we prove that
$\displaystyle \mathrm{m}_K(A,D) = \sum_{C \in C_{\mathcal{A}}^K} x_C \, \gamma_K(C)$ with $x_C \in B_{\mathcal{A}}^K$. 

\noindent We distinguish four kinds of chambers $C \in C_{\mathcal{A}}^K$ according to their frontiers:
\begin{enumerate}
\item if $\partial C$ is homeomorphic to $B^{n-1}$, we say that $C$ is of type $1$, and write $C \in C_{\mathcal{A}}^{K,1}$,
\item if $\partial C$ is homeomorphic to $\partial B^{n-1} \times \mathbb{R}$, we say that $C$ is of type $2$, and write $C \in C_{\mathcal{A}}^{K,2}$,
\item if $\partial C$ is homeomorphic to $\partial B^1 \times \mathbb{R}^{n-1}$, we say that $C$ is of type $3$, and write $C \in C_{\mathcal{A}}^{K,3}$,
\item if $\partial C$ is homeomorphic to $\partial B^n$, we say that $C$ is bounded, and write $C \in \breve{C}_{\mathcal{A}}^K$.
\end{enumerate}

\noindent Suppose that $D \in \breve{C}_{\mathcal{A}}^K \cup C_{\mathcal{A}}^{K,2} \cup C_{\mathcal{A}}^{K,3}$. With the same argument as in the proof of Theorem~\ref{D*}, we get $\displaystyle \sum_{F \in F_{\overline{D}}} (-1)^{\mathrm{rk}\,F} \mathrm{m}_K(F,D) = \pm D^*$, and then $\displaystyle D^* = \sum_{C \in C_{\mathcal{A}}^K} x_C \, \gamma_{\mathcal{A}}^K(C)$ with $x_C \in B_{\mathcal{A}}^K$.

\noindent Suppose now that $D \in C_{\mathcal{A}}^{K,1}$. Consider the hyperplane arrangement $\mathcal{A}' = \mathcal{A} \sqcup \{H'\}$ such that
\begin{itemize}
\item $H'$ divides $D$ into two chambers $D_b \in \breve{C}_{\mathcal{A}'}^K \cup C_{\mathcal{A}'}^{K,2} \cup C_{\mathcal{A}'}^{K,3}$ and $D_u \in C_{\mathcal{A}'}^{K,1}$,
\item for every $C \in \breve{C}_{\mathcal{A}}^K \cup C_{\mathcal{A}}^{K,2} \cup C_{\mathcal{A}}^{K,3}$, $H' \cap C = \emptyset$,
\item if $C_{\mathcal{A}}'^K$ is the set of chambers in $C_{\mathcal{A}}^K$ cut by $H'$, then, for every $C \in C_{\mathcal{A}}^K \setminus C_{\mathcal{A}}'^K$, $C \subset H'^+$,
\item if we denote by $C_b$ and $C_u$ the two chambers obtained from the cut of $C \in C_{\mathcal{A}}'^K$ by $H'$, then we assume that $C_b \subset H'^+$.
\end{itemize}
With an argument similar to the proof of Theorem~\ref{D*}, we obtain $$D_b^* = \sum_{C \in C_{\mathcal{A}}^K \setminus C_{\mathcal{A}}'^K} x_C \, \gamma_{\mathcal{A}}^K(C) + \sum_{C' \in C_{\mathcal{A}}'^K} x_{C'} \, \gamma_{\mathcal{A}}^K(C') = D^* \quad \text{with} \quad x_C, x_{C'} \in B_{\mathcal{A}}^K.$$

\noindent Finally, as $V_{\mathcal{A}}^K$ is the matrix representation of $\gamma_K$, one proves with an argument similar to the proof of Proposition \ref{fdetV} that $\det V_{\mathcal{A}}^K$ has the form $\displaystyle \prod_{F \in F_{\mathcal{A}}^K \setminus C_{\mathcal{A}}^K} (1 - \mathrm{b}_F)^{k_F}$. 
\end{proof}

\noindent The restriction of a hyperplane arrangement $\mathcal{A}$ to an apartment $K \in K_{\mathcal{A}}$ is the set arrangement $\mathcal{A}(K) := \{H \cap K\ |\ H \in \mathcal{A}:\, H \cap K \neq \emptyset\}$.

\noindent Let $F \in F_\mathcal{A}$, and $H \in \mathcal{A}$ such that $F \subseteq H$. Define the integer $\displaystyle \beta_F^H := \frac{\#\{C \in C_{\mathcal{A}}\ |\ \overline{C} \cap H = F\}}{2}$.
The following theorem not only computes $\det V_{\mathcal{A}}$, but also prove that for every hyperplane in the centralization of $F$, $\beta_F^H$ is the same. Hence the definition of the multiplicity.

\begin{theorem} \label{V}
Let $\mathcal{A}$ be a hyperplane arrangement in $\mathbb{R}^n$, and $F \in F_{\mathcal{A}}$. Then, $\beta_F^H$ has the same value $\beta_F$ for every $H \in \mathcal{A}_F$, and $$\det V_{\mathcal{A}} = \prod_{F \in F_{\mathcal{A}} \setminus C_{\mathcal{A}}} (1 - \mathrm{b}_F)^{\beta_F}.$$
\end{theorem}

\begin{proof}
From Proposition \ref{fdetV}, we have $\displaystyle \det V_{\mathcal{A}} = \prod_{F \in F_{\mathcal{A}} \setminus C_{\mathcal{A}}} (1 - \mathrm{b}_F)^{l_F}$. Take a face $E \in F_{\mathcal{A}} \setminus C_{\mathcal{A}}$: there exists an apartment $K \in K_{\mathcal{A}}$ such that $\mathcal{A}(K)$ is central with center $E$, and
$$\det V_{\mathcal{A}}^K = \prod_{F \in F_{\mathcal{A}}^K \setminus C_{\mathcal{A}}^K} (1 - \mathrm{b}_F)^{l_F'}.$$ Setting $h_H^+ = h_H^- = 0$ for every $H \in \mathcal{A} \setminus \mathcal{A}_E$, we see that, for every $F \in F_{\mathcal{A}}^K \setminus C_{\mathcal{A}}^K$, $l_F = l_F'$.

\noindent We prove by backward induction on the dimension of $E$ that $$\forall H,H' \in \mathcal{A}_E:\ \beta_E^H = \beta_E^{H'} = \beta_E \quad \text{and} \quad \det V_{\mathcal{A}}^K = \prod_{F \in F_{\mathcal{A}}^K \setminus C_{\mathcal{A}}^K} (1 - \mathrm{b}_F)^{\beta_F}.$$

\noindent Remark that $\displaystyle \beta_F^H = \frac{\#\{C \in C_{\mathcal{A}}^K\ |\ \overline{C} \cap H = F\}}{2}$. It is clear that, if $\dim E = n-1$, then $\beta_E = 1$ and $\det \mathcal{A}_E = 1 - \mathrm{b}_E$. If $\dim E < n-1$, by induction hypothesis, $$\det V_{\mathcal{A}}^K = (1 - \mathrm{b}_E)^{l_E} \, \prod_{F \, \in \, (F_{\mathcal{A}}^K \setminus C_{\mathcal{A}}^K) \setminus E} (1 - \mathrm{b}_F)^{\beta_F}.$$
Note that the leading monomial in $\det V_{\mathcal{A}}^K$ is $\displaystyle (-1)^{\frac{\#C_{\mathcal{A}}^K}{2}} \prod_{C \in C_{\mathcal{A}}^K} \mathrm{v}(C,\tilde{C}_E) \, = \, \big( -\prod_{H \in \mathcal{A}_E} h_H^+ \, h_H^- \big)^{\frac{\#C_{\mathcal{A}}^K}{2}}$.
Comparing the exponent of $h_H^+ \, h_H^-$, we get $\displaystyle l_E \, = \, \frac{\#C_{\mathcal{A}}^K}{2} - \sum_{\substack{F \, \in \, (F_{\mathcal{A}}^K \setminus C_{\mathcal{A}}^K) \setminus E \\ F \subseteq H}} \beta_F^H \, = \, \beta_E^H$.
\end{proof}

\noindent We can finally compute the Varchenko determinant for an apartment.

\begin{theorem}
Let $\mathcal{A}$ be a hyperplane arrangement in $\mathbb{R}^n$, and $K \in K_{\mathcal{A}}$. Then,
$$\det V_{\mathcal{A}}^K = \prod_{F \in F_{\mathcal{A}}^K \setminus C_{\mathcal{A}}^K} (1 - \mathrm{b}_F)^{\beta_F}.$$
\end{theorem}

\begin{proof}
From Proposition \ref{PdetK}, we have $\displaystyle \det V_{\mathcal{A}}^K = \prod_{F \in F_{\mathcal{A}}^K \setminus C_{\mathcal{A}}^K} (1 - \mathrm{b}_F)^{l_F}$. Take a face $E \in F_{\mathcal{A}}^K \setminus C_{\mathcal{A}}^K$: there exists an apartment $L \in K_{\mathcal{A}}$ such that $L \subseteq K$, $\mathcal{A}(L)$ is central with center $E$, and
$$\det V_{\mathcal{A}}^L = \prod_{F \in F_{\mathcal{A}}^L \setminus C_{\mathcal{A}}^L} (1 - \mathrm{b}_F)^{l_F'}.$$ Setting $h_H^+ = h_H^- = 0$ for every $H \in \mathcal{A} \setminus \mathcal{A}_E$, we see that, for every $F \in F_{\mathcal{A}}^L \setminus C_{\mathcal{A}}^L$, $l_F = l_F'$. Finally, we deduce from the proof of Theorem \ref{V} that $l_F = \beta_F$.
\end{proof}

\begin{remark}
Like mentioned by Hochstättler and Welker at the end of their article, a possible direction for generalizations is the computing of the Varchenko determinant of COMs. That would unify the Varchenko determinant for oriented matroids with use of the quantum bilinear form they computed \cite[Theorem~1]{HoWe} and Theorem~\ref{MTh2}. We would like to thank the reviewer who took the time to write detailed comments. Those were both really informative on the Varchenko determinant and very helpful for the presentation of this article.
\end{remark}

\bibliographystyle{abbrvnat}

\end{document}